\documentclass[12pt]{article}

\usepackage{times, amsmath, amsfonts, amssymb, amsthm}
\usepackage{times}
\usepackage{srcltx}
\usepackage{setspace}

\newcommand{\dd}{\text{\rm d}}             % a straight d for differentials
\newcommand{\ee}{\text{\rm e}}  
\newcommand{\cA}{\mathcal{A}}

\newcommand{\1}{{\bf{1}}}
\newcommand{\bigO}{\mathrm{O}}

\theoremstyle{plain}
\newtheorem{theorem}{Theorem}[section]
\newtheorem{lemma}[theorem]{Lemma}
\newtheorem{proposition}[theorem]{Proposition}

\theoremstyle{definition}

\newtheorem{remark}[theorem]{Remark}

\newtheorem{defi}[theorem]{Definition}

\renewenvironment{proof}[1][] {\noindent {\bf Proof#1.} }{\hspace*{\fill}$\square$\medskip\par}

\newcommand{\ve}{{\varepsilon}}

\newcommand{\N}{{\mathbf N}}

\newcommand{\R}{{\mathbf R}}
\newcommand{\Q}{{\mathbf Q}}
\newcommand{\E}{{\mathbf E}}
\newcommand{\F}{{\cal F}}

\newcommand{\A}{{\cal A}}

\newcommand{\NN}{{\cal N}}
\renewcommand{\P}{{\mathbf P}}

\newcommand{\tr}{{\mathrm {tr}}}
\author{Michael Scheutzow%
  \thanks{Institut f\"ur Mathematik, MA 7-5, Fakult\"at II, 
        Technische Universit\"at Berlin, 
        Stra\ss e des 17.~Juni 136, 10623 Berlin, FRG;  \ 
        \small{\tt ms{\scriptsize @}math.tu-berlin.de}}
\and Susanne Schulze%
  \thanks{Institut f\"ur Mathematik, MA 7-5, Fakult\"at II, 
        Technische Universit\"at Berlin, 
        Stra\ss e des 17.~Juni 136, 10623 Berlin, FRG;  \ 
        \small{\tt Susanne\underline{ }Richter85{\scriptsize @}web.de}}}

\title{Strong completeness and semi-flows for stochastic differential equations with monotone drift}

\date{\today}

\begin{document}  \maketitle

\begin{abstract}\noindent
It is well-known that a stochastic differential equation (sde) on a Euclidean space driven 
by a (possibly infinite-dimensional) Brownian motion with Lipschitz coefficients generates a stochastic flow of 
homeomorphisms. If the Lipschitz condition is replaced by an appropriate one-sided Lipschitz condition 
(sometimes called monotonicity condition) and the number of driving Brownian motions is finite, 
then existence and uniqueness of global solutions for each fixed 
initial condition is also well-known. In this paper we show that under a slightly stronger one-sided Lipschitz 
condition the solutions still generate a stochastic semiflow which is jointly continuous in all variables 
(but which is generally neither one-to-one nor onto). We also address the question of strong $\Delta$-completeness 
which means that there exists a modification of the solution which if restricted to any set $A \subset \R^d$ of 
dimension $\Delta$ is almost surely continuous with respect to the initial condition. 

  \par\medskip

\end{abstract}


\section{Introduction}\label{intro}
In this paper, we study properties of the stochastic differential equation (sde)
\begin{equation}
\label{sde}
\dd X_t=b(X_t)\,\dd t + M(\dd t,X_t),
\end{equation}
where $M$ is a continuous martingale field on $\R^d$. Under an appropriate Lipschitz condition, this sde has a unique solution
$X_t$, $t \ge s$ for each initial condition $X_s=x\in \R^d$ and, 
moreover, the sde generates a stochastic flow of homeomorphisms (\cite[Theorem 4.5.1]{K90}).  
The aim of this paper is to show that a weaker semiflow property still holds in case the Lipschitz condition is replaced by 
a (local) one-sided Lipschitz condition (also known as {\em monotonicity} condition) and a {\em coercivity} condition. Under these 
weaker conditions, we cannot expect to obtain a stochastic flow of homeomorphisms anymore: both the one-to-one property and 
the onto property may fail -- even in the deterministic case. The best we can hope for is a modification 
$\phi: \{0 \le s \le t <\infty\} \times \R^d \times \Omega \to \R^d$ of the solution which depends continuously on the 
temporal and the spatial variables and which enjoys the semiflow property.
We will provide sufficient conditions on $b$ and $M$ for this to hold. We believe that our approach provides a quick and transparant  
proof of quite general strong completeness results under rather weak conditions. In fact we do not even assume local Lipschitz continuity of the coefficients. 
We will compare our results with others at the end of this section.

We will denote the standard inner product on $\R^d$ by $\langle .,.\rangle$, the Euclidean norm on $\R^d$ by $|.|$, 
the induced norm on $\R^{d \times d}$ by $\|.\|$ (which is equal to the largest eigenvalue in case the 
matrix is positive semi-definite) and the (joint) quadratic variation of continuous 
semimartingales by $[.,.]$. We denote the trace of a matrix $A \in \R^{d \times d}$ by $\tr (A)$.
Throughout the paper, we will impose the following assumptions:
\begin{itemize}
\item[$\bullet$] $(\Omega,\F,(\F_t)_{t \ge 0},\P)$ is a filtered probability space satisfying the {\em usual conditions}.
\item[$\bullet$] $b:\R^d \to \R^d$ is continuous.
\item[$\bullet$] For each $x \in \R^d$, $t \mapsto M(t,x)$ is a continuous $\R^d$-valued martingale s.t. $M(0,x)=0$.
\item[$\bullet$] The matrix $a(x,y):=\frac{\dd}{\dd t}[M(.,x),M(.,y)]_t$ is non-random and independent of $t$. Further, 
the map $(x,y) \mapsto a(x,y)$ is continuous.
\item[$\bullet$] Define $\cA(x,y):=a(x,x)-a(x,y)-a(y,x)+a(y,y)$. 
For each $R>0$ there exists some $K_R \ge 0$ such that
$2 \langle b(x)-b(y),x-y \rangle + \tr (\cA(x,y)) \le K_R |x-y|^2$ for all $|x|,|y| \le R$.
\end{itemize}
Note that our assumptions imply that the field $(t,x) \mapsto M(t,x)$ is a centered Gaussian process and for each $x \in \R^d$, 
$t \mapsto M(t,x)$ has the same law as $t \mapsto GW_t$, where the $d \times d$ matrix $G$ satisfies $G G^T=a(x,x)$ and $W$ is $d$-dimensional 
Brownian motion. 
%In accordance with the conventions in the monograph \cite{K90}, we will call $a$ and $b$ 
%the {\em characteristics} of the SDE \eqref{sde}. 
We point out that the equation
$$
\dd X_t = b(X_t)\,\dd t + \sum_{k=1}^m  \sigma_k(X_t)\, \dd W^k_t
$$
with $W^1,...W^m$ independent standard Brownian motions is a special case of \eqref{sde} if we define 
$M(t,x):=\sum_{k=1}^m \sigma_k(x) W^k_t$. Then $a_{i,j}(x,y)=\sum_{k=1}^m  \sigma_k^i(x)\sigma_k^j(y)$ and
$\cA_{i,j}(x,y)= \sum_{k=1}^m ( \sigma_k^i(x)-   \sigma_k^i(y))( \sigma_k^j(x)-   \sigma_k^j(y))$. Note that 
we have (in general) $\cA(x,y)=\frac {\dd}{\dd t} [M(.,x)-M(.,y)]_t$.\\
%Sometimes, we will impose in addition one of the following stronger monotonicity or growth conditions:\\

%%%%%%%%%%%%%%%%%%%%%%%%%%%%%%%%%%%%%%%%%%%%%%%%%%%%%%%%%%%%%%%%%%%%%%%%%%%%%%%%%%%%%%%%%%%%%%%
%
%     The definition
%
%%%%%%%%%%%%%%%%%%%%%%%%%%%%%%%%%%%%%%%%%%%%%%%%%%%%%%%%%%%%%%%%%%%%%%%%%%%%%%%%%%%%%%%%%%%%%%%%

\begin{defi}\label{thedefinition}
\begin{itemize}
\item We say that \eqref{sde} has a {\em (strong) local solution} if for each $x \in \R^d$ and $s \ge 0$, 
there exists a stopping time 
$\tau:=\tau_{s,x}>s$ and an $\R^d$-valued adapted process $X_t,\,t\in [s,\tau)$ with continuous paths such that $X_t=x+\int_s^t b(X_u)\,\dd u 
+ \int_s^t M(\dd u,X_u)$ a.s. whenever $0<t<\tau$ and $\lim_{t \to \tau} |X_t|= \infty$ a.s. on the set $\{\tau<\infty\}$.
We say that the local solution is {\em unique}  if whenever  $\tilde X_t,\,t \ge s$ is another process with these properties with 
associated stopping time $\tilde \tau$, then $\tau=\tilde \tau$ and $X=\tilde X$ on $[s,\tau)$ almost surely. 
We will denote such a solution by $\phi_{s,t}(x)$. For $t \ge \tau$, we define $\phi_{s,t}(x):=\infty$.
\item We say that \eqref{sde} has a {\em unique global solution} or is {\em weakly complete} if it has a unique local solution and 
if for each $x \in \R^d$ and $s \ge 0$, $\tau=\infty$ almost surely. 
\item We say that   \eqref{sde} admits a {\em local semiflow}, if it has a unique local solution 
$\phi_{s,t}(x),\,t \in [s,\tau (s,x))$ 
which admits a modification 
$(\varphi_{s,t}(x), \Theta (s,x))$ which is a local semiflow, i.e.: $\varphi$ and 
$\Theta:[0,\infty)\times \R^d \times \Omega \to (0,\infty]$  are measurable and for each $\omega \in \Omega$,
\begin{itemize}
\item[i)] $(s,x) \mapsto \Theta(s,x)$ is lower semicontinuous
\item[ii)] $(s,t,x) \mapsto \varphi_{s,t}(x)$ is continuous on $\{(s,t,x):\,0\le s \le t < \Theta (s,x)\}$
\item[iii)] For all $0 \le s \le t \le u$, and $x \in \R^d$, we have 
$u<\Theta(s,x)$ iff both $t<\Theta(s,x)$ and $u < \Theta(t,\varphi_{s,t}(x))$ and in this case the following identity holds:
$\varphi_{s,u}(x)=\varphi_{t,u} ( \varphi_{s,t}(x))$
\item[iv)] $\varphi_{s,s}=\mathrm{id}_{\R^d}$ for all $s\ge 0$.
\item[v)] $\lim_{t \to \Theta(s,x)} |\varphi_{s,t}(x)|=\infty$ whenever $s \ge 0$, $x \in \R^d$, and $\Theta(s,x)<\infty.$
\end{itemize} 
\item We say that  \eqref{sde} admits a {\em global semiflow}, if it admits a local semiflow with $\Theta (s,x)=\infty$ for all 
$s\ge 0, \,x \in \R^d, \omega \in \Omega$.
\end{itemize}
\end{defi}

Note that for existence and uniqueness of local or global solutions it suffices to check the definition for $s=0$ due to time homogeneity of 
$a $ and $b$. We will also use the following definition.

\begin{defi}
We call $\phi:[0,\infty)\times \R^d \to \R^d \cup \{\infty\}$ a {\em continuous local map} if for each $x \in \R^d$, there exists 
some $\tau (x) \in (0,\infty]$ such that the following hold:
\begin{itemize}
\item[i)] $\phi_0 =\mathrm{id}|_{\R^d}$.
\item[ii)] $\phi_t(x)=\infty$ whenever $t \ge \tau(x)$.
\item[iii)] $\phi$ is (jointly) continuous with respect to the one-point compactification $\R^d \cup \{\infty\}$.
\end{itemize}
We call  $\phi:[0,\infty)\times \R^d \times \Omega  \to \R^d \cup \{\infty\}$ a {\em random continuous local map} in case 
$\phi$ is measurable and $\phi(.,\omega)$ is a continuous local map for every $\omega \in \Omega$. Further, we call 
$\phi$ a {\em continuous global map} resp.~{\em random continuous global map} if  $\phi$ is a continuous local map resp.~random 
continuous local map such that $\tau \equiv \infty$. 
\end{defi}

\begin{defi}
We say that the sde \eqref{sde} is {\em strongly complete} in case it has a unique local solution which for initial time $s=0$ admits a modification which is 
a random continuous global map. 
\end{defi}

Note that if the sde \eqref{sde} admits a global semiflow, then it is strongly complete.\\ 

%Note that a local (global) semiflow is 

%\noindent{\bf Assumption (L)}
%For each $N \in \N$ there exist $K_N \ge 0$ and $p_N >d+4$ %$p_N \ge 2$ 
%such that
%$2 \langle b(x)-b(y),x-y \rangle + \tr (\cA(x,y)) + (p_N-2)\|\A(x,y)\| \le K_N |x-y|^2$ for all $|x|,|y| \le N$.\\

%\noindent{\bf Assumption (L2)}
%There exist $\lambda \ge 0$ and $\sigma >0$ such that for all $x,y \in \R^d$ we have
%\begin{itemize}
%\item[i)] $2 \langle b(x)-b(y),x-y \rangle + \tr(\cA(x,y)) \le \lambda|x-y|^2$, 
%\item[ii)] $ \|\cA(x,y)\| \le \sigma^2 |x-y|^2$.\\
%\end{itemize}

In our main results, we will sometimes need the following assumptions. 
%Whenever we need them, we will explicitly state them.
In the following hypotheses, $\mu,K \ge 0$ and $\rho: [0,\infty) \to (0,\infty)$ is a non-decreasing function such that 
$\int_0^{\infty} 1/\rho(u)\, \dd u = \infty$.\\ 

\noindent{\bf Assumption (A$_{\mu,K}$)} 
%There exists $K \ge 0$ such that
$2 \langle b(x)-b(y),x-y \rangle + \tr(\cA(x,y)) + \mu \|\A(x,y)\| \le K |x-y|^2$ for all $x,y \in \R^d$.\\

We will show in the Appendix (Proposition \ref{locglob}) that the previous assumption holds if it holds locally, i.e. 
for each $z \in \R^d$ there is a neighborhood of $z$ such that the assumption holds for all $x,y$ in that neighborhood.\\

\noindent{\bf Assumption (A$_{\mu,{\rm{loc}}}$)} 
For each $R>0$ there exists $K_R \ge 0$ such that
$2 \langle b(x)-b(y),x-y \rangle + \tr(\cA(x,y)) + \mu\|\A(x,y)\| \le K_R |x-y|^2$ for all $|x|,|y| \le R$.\\

\noindent{\bf Assumption (G$_{\rho}$)}
For all $x \in \R^d$, we have
$$
2\langle b(x),x \rangle + \tr (a(x,x)) \le \rho(|x|^2).\\
$$ 

\noindent{\bf Assumption (G)} There exists a function $\rho$ as above, for which Assumption (G$_{\rho}$) holds.\\

%\noindent{\bf Assumption (G2)}
%There exist $K \ge 0$ and $p >d$, $p \ge 2$ such that
%$2 \langle b(x)-b(y),x-y \rangle + {\mbox {Tr}}(\cA(x,y)) + (p-2)\|\A(x,y)\| \le K |x-y|^2$ for all $x,y \in \R^d$.\\
%
%\noindent{\bf Assumption (G2')}
%There exist $K \ge 0$ and $p >d+4$  such that
%$2 \langle b(x)-b(y),x-y \rangle + {\mbox {Tr}}(\cA(x,y)) + (p-2)\|\A(x,y)\| \le K |x-y|^2$ for all $x,y \in \R^d$.\\

%\noindent{\bf Assumption (B)} 
%The functions $x \mapsto b(x)$ and $(x,y) \mapsto a(x,y)$ are bounded.\\

\noindent{\bf Assumption (H$_{f,\mu}$)}
$f:[0,\infty)\to (0,\infty)$ is continuous and nondecreasing, $\mu \ge 0$, and
$2 \langle b(x)-b(y),x-y \rangle + \tr(\cA(x,y)) + \mu\|\A(x,y)\| \le f(|x|\vee |y|) |x-y|^2$ 
for all $x,y \in \R^d$.\\

We will prove in the Appendix (Lemma \ref{G}) that (A$_{0,K}$) implies (G$_{\rho}$) when $\rho$ is a suitable multiple of $x \mapsto x \vee 1$ and hence 
(A$_{0,K}$) implies (G).\\ 

The paper is organized as follows. We will state sufficient conditions for existence and uniqueness of local and global solutions in Proposition 
\ref{solution}. 
We will state sufficient conditions for strong completeness in Theorem \ref{strongcompleteness}. Proposition \ref{examples} contains various explicit sufficient 
conditions for strong completeness. Theorems \ref{localflow} and \ref{globalflow}  provide sufficient conditions for the existence of a local, respectively global semiflow. 
In Section \ref{strongdelta} we define what we mean by strong $\Delta$-completeness and provide a sufficient condition for this property to hold (applying results of Ledoux and 
Talagrand on the existence of a continuous modification). In Section \ref{additivenoise} we consider the special case of additive noise in which case we can obtain better results 
(than the ones before if applied in that particular case).

Let us relate our results to prior work. Existence and uniqueness of solutions of the sde \eqref{sde} have been shown in \cite{PR07} (based on earlier work by Krylov in \cite{Kr99}) 
in case the sde is driven by a finite number of Brownian motions (they allow however random and time dependent coefficients). Our proof follows the one in \cite{PR07} initially 
but we apply a stochastic Gronwall lemma which simplifies the proof. The first major result about strong completeness of sdes is \cite{Li94}. Our results are more general in some sense 
(we do not require differentiability properties of the coefficents but just continuity and monotonicity) but less general in other respects (we only work on $\R^d$ instead of 
more general manifolds). We believe that our approach has the advantage of being straightforward and short (once existence and uniqueness of solutions and the stochastic 
Gronwall lemma are available).  We point out that \cite{FIZ07} contains strong completeness results under more restrictive conditions than ours. A recent paper dealing with strong completeness 
is \cite{CHJ13}. They only consider finite dimensional driving noise but they are more general in other respects (e.g.\ they consider Lyapunov functions while we only 
work with functions of the radial part of the solution). 

It has been observed before that in order to prove strong completeness, one needs to control both the growth of the driving vector fields and of the local Lipschitz constants 
which determine the local dispersion of the semiflow. Relaxing one of the conditions will typically require the other one to be strengthened (see Theorem 
\ref{strongcompleteness} and Proposition \ref{examples}). 
Even if the vector fields are bounded, a local Lipschitz condition  is insufficient for strong completeness as was shown in \cite{LS11}. The additive noise case is somewhat special. 
The correlation of the driving noise is such that the usual conditions on the drift (linear growth and local Lipschitz condition) suffice to show strong completeness (see 
Section \ref{additivenoise}). 
   
\section{Main results}

\begin{proposition}\label{solution}
\begin{itemize}
\item[a)] Equation \eqref{sde} has a unique local solution. Solutions enjoy the following coalescence property:
for each pair $x,y \in \R^d$, and $s,\,s' \ge 0$, the following holds true almost surely: if there exists $t \ge s\vee s'$ such that 
$\phi_{s,t}(x)=\phi_{s',t}(y)$, then $\phi_{s,u}(x)=\phi_{s',u}(y)$ for all $u \ge t$. 
\item[b)] If, moreover,  Assumption ${\mathrm{(G)}}$ holds, then equation \eqref{sde} has a unique global solution. 
\end{itemize}
\end{proposition}

\noindent We will prove Proposition \ref{solution} in the following section.

\begin{theorem}\label{strongcompleteness}
If  {\rm{(H$_{f,\mu}$)}} holds for some $\mu > d-2$ (and $\mu \ge 0$) and \eqref{sde} admits a global solution $\phi$ such that there exist $\gamma>0$ and $t_0>0$ 
such that for any $R>0$ we have
\begin{equation}\label{finite}
\sup_{|x|\le R}\sup_{s \in [0,t_0]}\E \ee^{\gamma f(|\phi_{0,s}(x)|)}< \infty
\end{equation}
then the sde is strongly complete. Moreover, if $(t,x) \mapsto \varphi_t(x)$ is a continuous modification of $(t,x) \mapsto \phi_{0,t}(x)$, then, for each $T>0$, 
the map $x \mapsto \varphi_.(x)$ from $\R^d$ to $C([0,T],\R^d)$  
is almost surely H\"older continuous with parameter $1-\frac dq$ for every $q \in (d,\mu +2)$.
\end{theorem}

We will prove Theorem \ref{strongcompleteness} and the following proposition in Section \ref{strong}. 
Next, we provide some examples for functions $f$ satisfying the assumptions of the previous corollary under suitable conditions on the 
coefficients $b$ and $a$ of the sde.

\begin{proposition}\label{examples} For each of the following combinations of $b,a$, and $f$, the assumptions of Theorem \ref{strongcompleteness} are satisfied and 
hence the sde is strongly complete. In each of the cases $\beta,c>0$ are arbitrary.

\begin{tabular}{lll}
\rm{a)} \hspace{.2cm} &$\langle b(x),x\rangle,\tr(a(x,x))\le c(1+|x|^2),\quad$ &$f(u)=\beta\big((\log^+u)^2+1\big)$\\
\rm{b)} &$2\langle b(x),x\rangle + \tr(a(x,x))\le c(1+|x|^2),\quad$ &$f(u)=\beta \big(\log^+u+1\big)$\\
\rm{c)} &$b,a$ bounded, &$f(u)=\beta \big(u^2+1\big)$.
\end{tabular}
\end{proposition}

\begin{theorem}\label{localflow}
Assume that  {\rm{(A$_{\mu,\mathrm{loc}}$)}} holds for some $\mu > d+2$. Then equation \eqref{sde} admits a local semiflow.
\end{theorem}

\begin{theorem}\label{globalflow}
Assume that unique global solutions of \eqref{sde} exist. If  {\rm{(H$_{f,\mu}$)}} holds for some $\mu > d+2$ and 
there exist $\gamma>0$ and $t_0>0$ such that for any $R>0$ we have
\begin{equation}\label{finiteneu}
\sup_{|x|\le R}\sup_{s \in [0,t_0]}\E \ee^{\gamma f(|\phi_{0,s}(x)|)}< \infty,
\end{equation}
then the sde admits a global semiflow.
\end{theorem}

We will prove Theorems   \ref{localflow} and  \ref{globalflow} in Section \ref{localandglobal}. We do not know if Theorem \ref{globalflow} 
remains true under the slightly weaker assumptions of Theorem \ref{strongcompleteness}

\begin{remark}
The case $d=1$ is special and in that case better results can be achieved due to the fact that $\R$ is totally ordered. In this case weak completeness 
plus existence of a local semiflow are (more than) enough to guarantee strong completeness (and even existence of a global semiflow). 
\end{remark}

\section{Existence and uniqueness of a local and global solution}

\begin{proof}[ of Proposition \ref{solution}]
We first show existence of a local solution. The first part of the proof is an adaptation of arguments in Chapter 3 of \cite{PR07} 
(based on previous work of Krylov \cite{Kr99}), while the final part is similar to a corresponding proof for stochastic functional 
differential equations (with finite dimensional noise) in 
\cite{RS10} (using a stochastic Gronwall lemma). Due to time-homogeneity of the coefficients, we can and will assume that $s=0$. Fix $x\in \R^d$.
Increasing the number $K_R$ if necessary, we can and will assume that $\sup_{|x|\le R}|b(x)|\le K_R$ for each $R>0$.
To prove existence of a solution, we employ an Euler scheme.
For $n \in \N$, we define the process $(\phi^{(n)}_t)_{t\in [0, \infty)}$ by
%\begin{flushleft}
$\phi^{(n)}_0 := x \in \R^d$
%\end{flushleft}
and -- for $ k \in \N_0 $ and $ t\in (\frac{k}{n},\frac{k+1}{n}]$ -- by
\begin{equation}\label{def1}
\phi^{(n)}_t := \phi^{(n)}_{\frac{k}{n}}\ + \int^{t}_{\frac{k}{n}} b(\phi^{(n)}_{\frac{k}{n}} )\,\dd s \ + 
\int^{t}_{\frac{k}{n}} M(\dd s,\phi^{(n)}_{\frac{k}{n}} ),
\end{equation}
which is equivalent to
\begin{equation}\label{def2}
	\phi^{(n)}_t = x \ + \int^{t}_{0} b(\bar{\phi}^{(n)}_{s})\,\dd s \ 
       + \int^{t}_{0} M(\dd s,\bar{\phi}^{(n)}_{s} )
\end{equation}
for $t\in [0 , \infty)$, where $\bar{\phi}^{(n)}_{s}:=\phi^{(n)}_{\frac{\left\lfloor ns\right\rfloor}{n}}$.
Defining $ p^{(n)}_t :=\bar{\phi}^{(n)}_{t}- \phi^{(n)}_{t}$, we obtain
\begin{equation}
	\phi^{(n)}_t = x \ + \int^{t}_{0} b(\phi^{(n)}_{s} + p^{(n)}_s)\,\dd s \ 
+ \int^{t}_{0} M(\dd s,\phi^{(n)}_{s} + p^{(n)}_s )
\end{equation}
for $t \in [0,\infty).$ Observe that $t \mapsto \phi^{(n)}_t $ is adapted and continuous.
%%%%%%%%%%%%%%%%%%%%%%%%%%%%%%%% Ito %%%%%%%%%%%%%%%%%%%%%%%%%%%%%%%%%%%%%%%%%%%%
Using It\^o's formula, we obtain for $t \ge 0$
\begin{align*}
&\big|\phi^{(n)}_t -\phi^{(m)}_t\big|^2\\
&= \int^t _0  2\left\langle \phi^{(n)}_s-\phi^{(m)}_s, b(\bar{\phi}^{(n)}_s)-b(\bar{\phi}^{(m)}_s)\right\rangle \dd s\\ 
&\ \ + \int^t_0  2\left\langle \phi^{(n)}_s -\phi^{(m)}_s,M(\dd s,\bar{\phi}^{(n)}_s)-M(\dd s,\bar{\phi}^{(m)}_s)\right\rangle
 + \int^t_0  \tr(\A(\bar{\phi}^{(n)}_s,\bar{\phi}^{(m)}_s))\,\dd s\\
%%%%%%%%%%%%%%%%%%%%%%%%%%%%%%%%%%%%%%%%%%%%%%%%%%%%%%%%%%Umformungen%%%%%%%%%%%%%%%%%%%%%%%%%%%%%%%%%%%%%%%%%%%%%%%
&=\int^t _0  2\left\langle \bar{\phi}^{(n)}_s-\bar{\phi}^{(m)}_s, b(\bar{\phi}^{(n)}_s)-b(\bar{\phi}^{(m)}_s)\right\rangle \dd s 
\ \ +\int^t_0 \tr(\A(\bar{\phi}^{(n)}_s,\bar{\phi}^{(m)}_s))\,\dd s\\ 
&\ \ + M^{(n,m)}_t
\ \ + \int^t_0 2\left\langle p^{(m)}_s- p^{(n)}_s, b(\bar{\phi}^{(n)}_s)-b(\bar{\phi}^{(m)}_s)\right\rangle \dd s,
\end{align*}
where
$$M^{(n,m)}_t:= \int^t_0  2\langle \phi^{(n)}_s -\phi^{(m)}_s, M(\dd s,\bar{\phi}^{(n)}_s)-M(\dd s,\bar{\phi}^{(m)}_s)\rangle$$
is a continuous local martingale starting at 0. 
%%%%%%%%%%%%%%%%%%%%%%%%%%%%%%%% Stoppzeiten Teil %%%%%%%%%%%%%%%%%%%%%%%%%
Let $R>3 |x|$ and define the stopping times
$$\tau^{(n)}(R):=\inf \left\{t\geq 0 : \big|\phi^{(n)}_t\big| \ge \frac{R}{3}\right\}.$$ 
Then
\begin{equation}\label{bounded}
\big|p^{(n)}_t\big|\leq \frac{2R}{3} \quad \text{ and }\quad \big|\phi^{(n)}_t\big|\leq \frac{R}{3}
 \quad \text{ for } t\in [0, \tau^{(n)}(R)]\cap [0,\infty).
\end{equation}
For $0\le s \le \tau^{(n)}(R)\wedge\tau^{(m)}(R)=:\gamma^{(n,m)}(R)$, we have
\begin{align*}
\langle p^{(m)}_s- p^{(n)}_s, b(\bar{\phi}^{(n)}_s)-b(\bar{\phi}^{(m)}_s)\rangle \leq 2K_R\left|p^{(m)}_s- p^{(n)}_s\right|
\le 2 K_R \big( |p_s^{(m)}| + |p_s^{(n)}| \big).
\end{align*}
Therefore, for $t \le \gamma^{(n,m)}(R)$, we get
\begin{align*}
&\big|\phi^{(n)}_{t} -\phi^{(m)}_{t}\big|^2\\
%%%%%%%%%%%%%%%%%%%%
\quad &\leq \int^{t}_0 (K_R\left|\bar{\phi}^{(n)}_s -\bar{\phi}^{(m)}_s\right|^2 
+ 2\left\langle p^{(m)}_s- p^{(n)}_s, b(\bar{\phi}^{(n)}_s)-b(\bar{\phi}^{(m)}_s)\right\rangle)\,\dd s +M^{(n,m)}_t\\
%%%%%%%%%%%%%%%%%%%%
\quad &\leq \int^{t}_0  2K_R(\left|\phi^{(n)}_s-\phi^{(m)}_s\right|^2 +\left|p^{(n)}_s-p^{(m)}_s\right|^2) 
+ 4K_R(\left|p^{(n)}_s\right|+\left|p^{(m)}_s\right|)\,\dd s  + M^{(n,m)}_t\\
%%%%%%%%%%%%%%%%%%%%%
 &\leq\int^{t}_0  2K_R  \left|\phi^{(n)}_s-\phi^{(m)}_s\right|^2\dd s
 + 4\int^{t}_0  K_R(\left|p^{(n)}_s\right|+\left|p^{(m)}_s\right|+\left|p^{(n)}_s\right|^2+ \left|p^{(m)}_s\right|^2)\,\dd s
+ M^{(n,m)}_t.
\end{align*}
%%%%%%%%%%%%%%%%%%%%%%%%%%
Now, we apply Lemma \ref{thelem} to the process $Z_t:=\big|\phi^{(n)}_{t\wedge \gamma^{(n,m)}(R)} -\phi^{(m)}_{t\wedge \gamma^{(n,m)}(R)}\big|^2$. 
Note that the assumptions are satisfied with $\psi:=2 K_R$, 
$M_t:=M^{(n,m)}_{t\wedge \gamma^{(n,m)}(R)}$, and 
$$
H_t:= H^{n,m}_t:=4\int^{t\wedge \gamma^{(n,m)}(R)}_0  K_R(\left|p^{(n)}_s\right|+\left|p^{(m)}_s\right|+\left|p^{(n)}_s\right|^2+
\left|p^{(m)}_s\right|^2)\,\dd s.$$ 
Therefore, for $p \in (0,1)$ and all $T >0$, we have 
\begin{align}\label{esti}
\begin{split}
	\E \left[\sup_{t\in [0,T]}\left|\phi^{(n)}_{t\wedge \gamma^{(n,m)}(R)} -\phi^{(m)}_{t\wedge \gamma^{(n,m)}(R)}\right|^{2p}\right]
&\leq \tilde c_p\ee^{2p K_R T} \Big(\E\Big[\sup_{t\in [0,T]} \big(H^{n,m}_t\big)^p\Big]\Big)\\ 
&=  \tilde c_p\ee^{2p K_R T}          \left(\E\left[H^{n,m}_T\right]^p\right).
\end{split}
\end{align}
It is easy to check that $\sup_{s \ge 0}\E|p^{(n)}_{s\wedge \tau^{(n)}(R)}|$ converges to 0 as $n \to \infty$ since the coefficients of the sde 
are bounded on bounded subsets of $\R^d$. Thanks to \eqref{bounded}, the  
right hand side of \eqref{esti} therefore converges to 0 as $n,m \rightarrow \infty$.
%%%%%%%%%%%%%%%%%%%%%%%%%%%%%%%%%%%%%%%%%%%%%%%%%%%%%%%%%%%%%%%%%%%%%%%%%%%%%%%%%%%%%%%%%%%%%%%%%%%%%%%%%%%%%%%%%%%%%%
%%%%%%%%%%%%%%%%%% H(t) geht im erwartungswert gegen null für n,m gegen infty
%%%%%%%%%%%%%%%%%%%%%%%%%%%%%%%%%%%%%%%%%%%%%%%%%%%%%%%%%%%%%%%%%%%%%%%%%%%%%%%%%%%%%%%%%%%%%%%%%%%%%%%%%%%%%%%%%%%%%
Now, by standard arguments, we can find a stopping time $\tau_R>0$, a process $\phi_t, \,t \in [0,\tau(R)] \cap [0,\infty)$, 
such that $\tau_R = \inf\{t \ge 0: |\phi_t|=R\}$  such that $\phi^{n}\to \phi_t$ on 
$[0,\tau_R] \cap [0,t]$ uniformly in probability for each $t>0$. 
Further, $\phi$ solves the  sde  \eqref{sde} with initial condition $x$ 
on the interval $[0,\tau_R] \cap [0,\infty)$ (see \cite{RS10} for 
details -- in the more general case of a stochastic delay differential equation).

Assume that there is another solution $\tilde \phi_t, \,t \in [0,\tilde \tau_R] \cap [0,\infty)$ with associated stopping time
$\tilde \tau=\inf\{t \ge 0: |\tilde \phi_t|=R\}$ which solves  \eqref{sde} with initial condition $x$ 
on the interval $[0,\tau_R] \cap [0,\infty)$.
Then we can apply It\^o's formula to the square of the norm of the difference of the 
two processes up to the minimum of the two stopping times and use Lemma \ref{thelem} to see that both solutions (and the associated 
stopping times) agree almost surely.

Now we let $R \to \infty$. Define $\tau:=\lim_{R \to \infty} \tau_R$. Then the construction above shows that there exists a unique 
solution $\phi$ of  \eqref{sde} with initial condition $x$ 
on the interval $[0,\tau)$ and that $\tau$ has all properties stated in the first part of Definition \ref{thedefinition}.

Now we show the coalescence property. Let $x,y \in \R^d$ and $0 \le s \le s'$. Define the stopping time 
$T:=\inf\{t \ge s':\phi_{s,t}(x)=\phi_{s',t}(y)\}$. 
On the set where $\phi_{s,T}(x)=\phi_{s',T}(y)=\infty$ and on the set $\{T=\infty\}$ there is nothing to prove. Therefore, we define
$$
Z_t:=|\phi_{s,T+t}(x)-\phi_{s',T+t}(y)|^2 \1_{\{T<\tau_{s,x}\wedge \tau_{s',y}\}},\quad t \ge 0.
$$
Applying Ito's formula and Lemma \ref{thelem} (as above), we see that $Z \equiv 0$,
so the proof of part a) is complete.\\

%%%%%%%%%%%%%%%%%%%%%%%%%    Part b)      %%%%%%%%%%%%%%%%%%%%%%%%%%%%%%%%%%%%%%%%%%%%%%%%%%%%%%%%%%%%%%%%%%%%%%%%%%%%%%

Next, we prove part b) of the proposition. W.l.o.g.~we assume that $s=0$. Let $(X,\sigma)$ be a maximal strong solution of 
the sde \eqref{sde} starting at $x$. We want to show that $\sigma=\infty$ almost surely. Note that 
$\limsup_{t \nearrow \sigma} |X_t|=\infty$ almost surely on the set $\{\sigma<\infty\}$. For a stopping time $0 \le \tau < \sigma$, 
It\^o's formula implies that
\begin{align}\label{square}
\begin{split}
X^2_{\tau}-X^2_0
&=\int_0^{\tau} 2\langle b(X_u),X_u\rangle + \tr(a(X_u,X_u))\, \dd u +
2 \int_0^{\tau} \langle X_u,M(\dd u,X_u) \rangle\\
&\le \int_0^{\tau} \rho(|X_u|^2)\, \dd u + \tilde M_\tau,
\end{split}
\end{align}
where $\tilde M$ is a continuous local martingale. Applying Lemma \ref{thefirstlemma} to $Z_t:=|X_t|^2$ finishes the proof. 
\end{proof}

The following proposition is a straightforward consequence of the uniqueness of local solutions.

\begin{proposition}\label{composition}
Under the assumptions of part a) of the previous proposition, the following holds for any modification of the (unique) local 
solution $(\phi,\tau)$: for each $0 \le s \le t \le u$ and $x \in \R^d$, there exists 
a set of full measure $\Omega_0$, such that on the set $\Omega_0$ the following holds:
\begin{itemize}
\item [$\bullet$] $u<\tau(s,x)$ iff both $t<\tau(s,x)$ and $u<\tau(t,\phi_{s,t}(x))$ 
\item [$\bullet$]
$\phi_{s,u}(x,\omega)=\phi_{t,u}(\phi_{s,t}(x,\omega),\omega)$ whenever $u<\tau(s,x)$.
\end{itemize}
\end{proposition}

%\begin{proof}
%Das ist ganz einfach -- sollte man aber noch etwas zu sagen!
%\end{proof}

\section{Strong completeness}\label{strong}
Our next aim is to establish sufficient conditions for strong completeness of an sde. We will show the 
existence of a (H\"older) continuous modification with the help of Kolmogorov's continuity theorem. Therefore, 
we will start by providing suitable $L^p$-estimates for the difference of solutions with different initial conditions.

\begin{lemma}\label{ganznewestimate}
Let $p \ge 2$ and assume that {\bf (H$_{f,p-2}$)} holds and that global solutions exist (for which {\bf (G)} is sufficient). 
Further, let $0<q<p$ and $P,Q > 1$ be such that $\frac 1P + \frac 1Q =1$ and $qQ/p<1$. Then, for each 
$0 \le s \le T$,
$$
\E\sup_{s \le t \le T}|\phi_{s,t}(x)-\phi_{s,t}(y)|^q \le 
|x-y|^q   \tilde c_{qQ/p}^{1/Q}    %\left( 1-Q \frac qp    \right)^{-1/Q}
\Big( \E \exp \Big\{\frac  {Pq}{2}\int_s^T  f(|\phi_{s,u}(x)| \vee |\phi_{s,u}(y)|) \dd u \Big\} \Big) ^{1/P}, 
$$
where the constant $\tilde c_r$ is defined before Lemma \ref{thelem}.
\end{lemma}

\begin{proof}
It suffices to prove the lemma in case $s=0$. Fix $x,\,y \in \R^d$, $x \neq y$.
Define
$$D_{t} := \phi_t(x) - \phi_t(y), \, Z_t:=|D_{t}|^p.$$
Then, by It\^o's formula, 
\begin{align*}
\dd Z_t&=p|D_t|^{p-2}\langle b(\phi_t(x))-b(\phi_t(y)), D_t \rangle\, \dd t + p|D_t|^{p-2}\langle D_t,M(\dd t,\phi_t(x))-M(\dd t,\phi_t(y))\rangle\\
&+ \frac 12 p |D_t|^{p-2}\tr (\A(\phi_t(x),\phi_t(y))) \,\dd t 
+ \frac 12 p(p-2)|D_t|^{p-4}\langle D_t, \A(\phi_t(x),\phi_t(y)) D_t\rangle\, \dd t,
\end{align*}
where the last term should be interpreted as zero when $D_t=0$ even if $p<4$. Therefore, using (H$_{f,p-2}$), we get
$$
Z_{t} \le |x-y|^{p} + \frac p2\int_0^t Z_u f(|\phi_u(x)| \vee |\phi_u(y)|)\,  \dd u + N_t,
$$
where $N$ is a continuous local martingale starting at 0. Lemma \ref{thelem} implies
$$
\E \sup_{0 \le s \le t}Z_s^r \le |x-y|^{rp}  \tilde c_{Qr}^{1/Q} %(1-Qr)^{-1/Q} 
\big(\E \ee^{Pr\int_0^t\psi_u \dd u}\big)^{1/P}, 
$$
where $\psi_u:= \frac p2 f(|\phi_u(x)| \vee |\phi_u(y)|)$, $r:=q/p$, $P,Q>1$, $rQ<1$, and $\frac 1P + \frac 1 Q=1$, 
so the assertion of the lemma follows.
\end{proof}

\begin{remark}\label{remark}
Clearly, Lemma \ref{ganznewestimate} remains true if the expectations are replaced by the conditional expectations given 
$\F_s$ since $\phi$ has independent increments.
\end{remark}

%\begin{corollary}\label{coronewestimate}
%Let Assumption ${\mathrm{(G2)}}$ hold. For each $s \ge 0$, there exists a modification $\varphi^{(s)}(t,x)$ of 
%$(t,x) \mapsto \phi_{s,t}(x), \; t \ge s$ which is continuous.
%\end{corollary}
%
%\begin{proof}
%This follows by applying Kolmogorov's continuity theorem and the previous lemma with $q \in (d,p)$.
%\end{proof}

\begin{proof}[ of Theorem \ref{strongcompleteness}]
Fix $p,q,P,Q$ as in Lemma \ref{ganznewestimate}. Using Jensen's inequality, we get for $t>0$
\begin{align*}
\sup_{|x|,|y|\le R}\E \exp\Big\{ \frac{Pq}2 &\int_0^t f(|\phi_{u}(x)|\vee|\phi_{u}(y)|)\,\dd u \Big\}\\ 
&\le \sup_{|x|,|y|\le R} \sup_{u \in [0,t]} \E \exp\Big\{ \frac{Pq}2 t  f(|\phi_{u}(x)|\vee|\phi_{u}(y)|)\Big\}\\ 
&\le 2\sup_{|x|\le R} \sup_{u \in [0,t]} \E \exp\Big\{ \frac{Pq}2 t  f(|\phi_{u}(x)|)\Big\}.\\ 
\end{align*}
Choosing $t>0$ sufficiently small, the right hand side is finite for any choice of $R$. Using Lemma \ref{ganznewestimate} it follows from Kolmogorov's continuity theorem with $q \in (d,p)$ 
that there exists a modification $\varphi$ of $\phi$ which is continuous on $\R^d \times [0,t_1]$ for some $t_1>0$ and that this modification  has the 
stated H\"older continuity property. Iterating, we obtain a H\"older continuous modification on  $\R^d \times [0,\infty)$. 
\end{proof}

\begin{proof}[ of Proposition \ref{examples}]
To show b), we use the first equality in \eqref{square} and then apply Lemma \ref{G} to get $\sup_{|x| \le R} \sup_{0 \le s \le t}|\phi_{0,s}(x)|^p<\infty$ for every 
$R,t >0$ and $p\in (0,2)$. Therefore b) follows. 

To show a), apply It\^o's formula to $Y_t:=\log \big(|X_t|^2 + 1\big)$ and use  the first equality in \eqref{square} to see that $Y$ has Gaussian tails  
uniformly on compact subsets of $\R^d \times [0,\infty)$. Therefore, a) follows.

To show c),  apply It\^o's formula to $Y_t:=\big(|X_t|^2 + 1)^{1/2}$  and use  the first equality in \eqref{square} to see that $Y$ has Gaussian tails     
uniformly on compact subsets of $\R^d \times [0,\infty)$. Therefore, c) follows.
\end{proof}

\section{Local and global semiflows}\label{localandglobal}

%From now on -- until the end of the section -- we will assume that both (G2') and  (G3) hold. Further, 
%we fix $q \in (d+4,p)$.

\begin{lemma}\label{lp} 
Assume that {\rm{(A$_{\mu,K}$)}} holds for some $\mu,K\ge 0$ and that $a$ and $b$ are globally bounded.
For each $T>0$ and $q\in (\mu + 2)$, there exists a constant $c \ge 0$ such that for all $x,y \in \R^d$ and all $0 \le s \le t\le T$, 
$0 \le s' \le t'\le T$, we have
$$ 
\E \big(|\phi_{s,t}(x)-\phi_{s',t'}(y)|^{q}\big) \le c\,\big(|x-y|^{q} + |t'-t|^{q/2} + |s'-s|^{q/2}\big).
$$
\end{lemma}

\begin{proof} Fix $T>0$. We will assume without loss of generality that $0 \le s \le t \le t' \le T$ and 
$0 \le s' \le t'\le T$. Further, $c$ denotes a constant (possibly depending on $a,b,\mu,K,q$ and $T$) whose value may change from line to line.
First note that Lemma \ref{G} implies weak completeness. 
\begin{align}\label{lpabsch}
\begin{split}
\E \big(|\phi_{s,t}(x)-&\phi_{s',t'}(y)|^{q}\big)\le c\, \Big( \E \big(|\phi_{s,t}(x)-\phi_{s,t}(y)|^{q}\big)\\
&+ \E \big(|\phi_{s,t}(y)-\phi_{s,t'}(y)|^{q}\big)
+ \E \big(|\phi_{s,t'}(y)-\phi_{s',t'}(y)|^{q}\big)\Big).
\end{split}
\end{align}
We estimate the three terms separately. Concerning the first term, Lemma \ref{ganznewestimate} (with $f$ equal to the constant $K$) implies
$$
\E \big(|\phi_{s,t}(x)-\phi_{s,t}(y)|^{q}\big) \le c |x-y|^{q}.
$$
Applying Burkholder's inequality, the second term in  \eqref{lpabsch} can 
be estimated as follows:
\begin{equation}\label{second}
\E \big(|\phi_{s,t}(y)-\phi_{s,t'}(y)|^{q}\big)=\E \Big| \int_t^{t'}b(\phi_{s,u}(y))\,\dd y +  \int_t^{t'} 
M\big(\phi_{s,u}(y),\dd u\big)\Big|^{q}\le c|t'-t|^{q/2}.
\end{equation}
Finally, we estimate the third term in \eqref{lpabsch}. Assuming w.l.o.g.\ that $s \le s'$, we get
\begin{eqnarray*}
 \E \big(|\phi_{s,t'}(y)-\phi_{s',t'}(y)|^{q}\big) &=& \E \big(|\phi_{s',t'}\big(\phi_{s,s'}(y)\big)-\phi_{s',t'}(y)|^{q}\big)\\
 &=& \E \big(\E\big(|\phi_{s',t'}\big(\phi_{s,s'}(y)\big)-\phi_{s',t'}(y)|^{q}|\F_{s'}\big)\big)\\
&\le& c\,\E |\phi_{s,s'}(y)-y|^{q}, 
\end{eqnarray*}
where we used Lemma \ref{ganznewestimate} and Remark \ref{remark} in the last step. The last term can be estimated by
$c \, |s'-s|^{q/2}$ just like in \eqref{second}. 
Therefore, the assertion of the lemma follows.
\end{proof}

We continue by proving a version of Theorem \ref{globalflow} under stronger assumptions. 

\begin{proposition}
Let the assumptions of the previous lemma hold for some $\mu>d+2$. Then the sde \eqref{sde} admits a global semiflow.
\end{proposition}

\begin{proof}
Lemma \ref{lp} and Kolmogorov's continuity theorem (as formulated e.g. in \cite{K90}, Theorem 1.4.1) show that there exists a 
modification $\varphi$ of the solutions $\phi$ 
which is jointly continuous in all three variables for all $\omega \in \Omega$. Since $\varphi_{s,s}(x)=\phi_{s,s}(x)=x$ 
almost surely for each fixed $s$ and $x$ and since $(s,x) \mapsto \varphi_{s,s}(x)$ is 
continuous, there exists a set $\NN$ of measure 0 such that  $\varphi_{s,s}(x)=x$ for all $x,s$ and all $\omega \notin \NN$. 
Redefining $\varphi_{s,t}(x):=x$ on $\NN$, we obtain   $\varphi_{s,s}(x)=x$ for all $s,x,\omega$.

It remains to show that for each $0 \le s \le t \le u$, $x \in \R^d$, and $\omega \in \Omega$, we have
\begin{equation}\label{semiflow}
\varphi_{s,u}(x)=\varphi_{t,u} (\varphi_{s,t}(x)).
\end{equation}
Observe that \eqref{semiflow} holds up to a null set depending on $s,t,u,$ and $x$ by Proposition \ref{composition}.   
Since both sides of \eqref{semiflow} are continuous functions of $(s,t,u,x)$, 
we can find a null set $\widetilde \NN$ in $\Omega$ such that 
$\varphi_{s,u}(x)=\varphi_{t,u} (\varphi_{s,t}(x))$ holds for all $(s,t,u,x)$ outside of $\widetilde \NN$. Redefining 
$\varphi_{s,t}(x):=x$ on $\widetilde \NN$, we see that $\varphi$ is a global semiflow associated to \eqref{sde}.
\end{proof}

Next, we relax the assumptions in the previous section and provide sufficient conditions for a local and a global semiflow to exist.\\

\begin{proof}[ of Theorem \ref{localflow}]
Let Assumption $\mathrm{(A_{\mu,\rm{loc}})}$ be satisfied and let $N \in \N$. We can find a function $b^N: \R^d \to \R^d$ and a continuous martingale field 
$M^N$ on the given filtered probability space such that $b^N(x)=b(x)$ and $M^N(t,x)=M(t,x)$  for all $|x|\le N$ and all $t \ge 0$ and 
such that $b^N$ and $a^N$ corresponding to  $M^N$ satisfy assumptions $\mathrm{(A_{\mu,K})}$ for some  $K \ge 0$ (depending on $N$) and are globally bounded. 
For example, we can take any function 
$\psi: \R \to [0,1]$ which is $C^{\infty}$ and non-increasing and which satisfies $\psi(s)=1$ for $s \le 1$ and $\psi(s)=0$ for $s \ge 2$ and 
define 
$$
b^N(x):=\Big( \psi\Big( \frac{|x|}{N} \Big)\Big)^2\,b(x),\qquad  M^N(t,x):=\psi\Big( \frac{|x|}{N} \Big)\,M(t,x),\;x \in \R^d,\,t \ge 0.
$$ 
By Theorem \ref{globalflow}, there exists an 
associated global semiflow $\varphi^N$. Clearly, there exists a set $\Omega_0$ of full measure such that for all $\omega \in \Omega$, and 
all $N>M$, $M,N \in \N$ the semiflows $\varphi^M$ and $\varphi^N$ agree inside a ball of radius $M$ in the following sense:
For all $0 \le s \le u$, $x \in \R^d$, we have $\sup_{s \le t \le u}|\varphi_{s,t}^M(x)|<M$ iff  $\sup_{s \le t \le u}|\varphi_{s,t}^N(x)|<M$ 
and in this case $\varphi_{s,u}^M(x)= \varphi_{s,u}^M(x)$. For $\omega \in \Omega_0$, we define 
$\Theta(s,x,\omega):=\lim_{N \to \infty} \inf\{t \ge s: |\varphi_{s,t}^N(x)| \ge N\}$ 
and for $t \in [s,\Theta(s,x,\omega))$, define $\varphi_{s,t}(x,\omega):=\varphi_{s,t}^N(x,\omega)$, where $N$ is any positive integer 
satisfying $\sup_{s \le u \le t}|\varphi^N_{s,u}(x)|<N$ (such an $N$ exists and the definition is independent of the choice of $N$). 
On the complement of $\Omega_0$, we define 
$\Theta \equiv \infty$ and $\varphi \equiv {\mathrm {id}}_{\R^d}$. It is straightforward to check that $(\varphi,\,\Theta)$ 
is a local semiflow of equation \eqref{sde}.
\end{proof}

\begin{proof}[ of Theorem \ref{globalflow}]
From the previous section, we know that \eqref{sde} admits a local semiflow $\varphi$ and 
Theorem \ref{strongcompleteness} shows that for every $s \ge 0$, there exists a set $\Omega_s$ of full measure such that 
$\Theta(s,x,\omega)=\infty$ for all $x \in \R^d$ and all $\omega \in \Omega_s$. Let $\widetilde \Omega$ be 
the intersection of the sets $\Omega_s$ for all rationals $s\ge 0$ and redefine $\varphi$ as the identity on the complement 
of $\widetilde \Omega$. Then we have $\Theta(s,x,\omega)=\infty$ for all $x \in \R^d$, $s \ge 0$,  and all $\omega \in \Omega$, 
i.e.~$\varphi$ is a global semiflow.
\end{proof}

\section{Strong $\Delta$-completeness}\label{strongdelta}

So far, we have discussed weak and strong completeness of an sde which mean that images of single points, respectively, subsets of full dimension 
survive almost surely under a  locally continuous modification of the solution. One can consider intermediate concepts of completeness. The concept of {\em strong $p$-completeness} for 
integer $p$ was introduced in \cite{Li94} meaning that $p$-dimensional submanifolds survive under the local semiflow. It seems natural to consider 
a corresponding concept of strong $\Delta$-completeness for arbitrary subsets of dimension $\Delta \in [0,d]$. For a precise definition, we need to 
agree on a particular notion of dimension like Hausdorff dimension or upper Minkowski dimension. In the following definition we will choose the upper 
Minkowski dimension (also known as box (counting) dimension) only because we can prove some result for it. We will in fact not even assume that a local 
semiflow exists.

\begin{defi}
Let $\Delta \in [0,d]$. We say that the sde \eqref{sde} is {\em strongly $\Delta$-complete} if for any (deterministic) set $A \subset \R^d$ of upper Minkowski 
dimension $\Delta$ there exist a set $\Omega_0$ of $\Omega$ of full measure and a modification $\varphi$ of the solution $\phi$ starting at time 0 for which 
$(t,x) \mapsto \varphi_t(x)$ is continuous on $[0,\infty) \times A$ for all $\omega \in \Omega_0$.  
\end{defi}

\begin{proposition}
Assume that {\rm (A$_{\mu,K}$)} holds. Then the sde \eqref{sde} is strongly $\Delta$-complete for each $\Delta < \mu +2$ (satisfying $\Delta \in [0,d]$).
\end{proposition}

\begin{proof}
For $q \in (\Delta,\mu+2)$ and $T>0$, Lemma \ref{G} implies weak completeness and Lemma \ref{ganznewestimate} implies that 
$$
\E \sup_{0 \le t \le T} |\phi_t(x)-\phi_t(y)| \le |x-y|^q c \exp\{qKT/2\},
$$
for some constant $c$ and all $x,y \in \R^d$. Now, a combination of Theorems 11.1 and 11.6 in \cite{LT91} applied to a set $A$ in $\R^d$ of upper 
Minkowski dimension $\Delta$ implies our claim. 
\end{proof}

\begin{remark}
The image under $\varphi_t$ of the set $A$ in the proof of the previous proposition is even almost surely bounded for each $t >0$ (this follows from 
the same theorems in \cite{LT91}).
\end{remark}

\section{Additive noise}\label{additivenoise}

Consider the sde
\begin{equation}\label{additive}
\dd X_t=b(X_t)\,\dd t + \sigma \dd W_t,
\end{equation}
where $b:\R^d \to \R^d$, $\sigma>0$ and $W$ is a $d$-dimensional Wiener process and $b$ satisfies our standing assumption. If $b$ has linear growth, i.e. there exists 
some $c\ge 0$ such that $|b(x)| \le c(1+|x|)$, then the flow generated by the sde is strongly complete, since for each initial condition $x$, we have 
$$
|X_t(x)| \le |x|+c\int_0^t(|X_s(x)|+1)\,\dd s +\sigma |W_t|
$$
for every $t >0$ and an application of Gronwall's lemma yields
$$
|X_t(x)| \le \big(|x|+\sigma \sup_{0\le s \le t}|W_s|+ct\big) \exp\{ct\}.
$$
This is of course well-known, see \cite{CHJ13} also for a discussion in case $b$ does not have linear growth.

One might be tempted to conjecture that one can replace the linear growth property of $b$ by the slightly weaker property
\begin{equation}\label{onesidedgrowth}
\langle b(x),x \rangle \le c(1+|x|^2)
\end{equation}
for some $c \ge 0$ and all $x \in \R^d$ but this does not seem to be true -- not even in case $c=0$. Instead of providing an example we 
just indicate how an example could look like (without making any claims that these ideas can be turned into rigorous mathematics):  

Let $d=2$, $\sigma=1$ and let $\rho: [0,\infty) \to \R$ be a smooth function such that $\rho(0)=0$ and $\rho$ has heavy and increasingly 
quick oscillations. Consider $b(x_1,x_2):=\rho(|x|){ {-x_2} \choose {x_1}}/|x|$ for $x \neq 0$. Then $\langle b(x),x \rangle = 0$ 
for all $x$. Assume we observe the motion of the unit ball under the flow. If -- in a short time interval -- the ball is pushed a bit 
in the positive direction of the first coordinate (say), then the huge tangential drift will ensure that the expansion in the 
negative first coordinate direction is (at least) almost as large as in the first. 
So the outer boundary of the image of the ball will remain 
to look almost like a sphere with center 0 and the radius will on average  increase at least by order $\delta^{1/2}$ 
during a time interval of length $\delta$. Since $\delta$ can be chosen arbitrarily small, it follows that strong completeness 
cannot hold.

These considerations suggest that in addition to \eqref{onesidedgrowth} one should impose some control on the growth of the tangential part 
of $b$. The following proposition shows that a quadratic (not just linear!) growth of that component guarantees strong completeness in the additive noise case.

\begin{proposition}\label{addinoise}
Let $b: \R^d \to \R^d$ satisfy our standing assumptions and assume that -- in addition -- there exists some $c \ge 0$ such that
for all $x \in \R^d$ the following hold true:
\begin{itemize}
\item[(i)] $\langle b(x),x\rangle  \le c(1+|x|^2)$,
\item[(ii)] $\Big| b(x)-\frac{\langle b(x),x\rangle}{|x|^2}x \Big| \le c(1+|x|^2)$.
\end{itemize}
Then, for each $\sigma \ge 0$, the sde is \eqref{additive} is strongly complete (and even admits a global flow).
\end{proposition}

\begin{proof}
For $x \in \R^d$ let $X_t(x),t \ge 0$ be the solution of \eqref{additive} with initial condition $x$ and define  
$Y_t(x):=X_t(x)-\sigma W_t$. Then
\begin{equation}\label{estima}
\frac{\dd}{\dd t} |Y_t(x)|^2=2\langle Y_t(x),b(Y_t(x)+\sigma W_t)\rangle.
\end{equation}
Fix $T >0$ and $\omega \in \Omega$ and let $u \in \R^d$ be such that $|u|\le \sigma \sup_{0 \le t \le T}|W_t|$. 
Once we succeed in  showing  that there exists some $C=C(T,\omega,\sigma,c)$ such that
\begin{equation}\label{to show}
\langle y,b(y+u)\rangle \le C (1+|y|^2)
\end{equation}
for all $y \in \R^d$ and $u$ as above, then the claim in the proposition follows by applying Gronwall's lemma to \eqref{estima}. 
Fix $y$ and $u$ as above such that $|y|\ge \sigma \sup_{0 \le t \le T}|W_t| +1$ (there is no need to consider smaller $|y|$). 
Then $b(y+u)$ can be uniquely decomposed as 
$$
b(y+u)=\alpha (y+u) + \beta v,
$$ 
where $\alpha,\beta \in \R$ and $v$ is a unit vector which is orthogonal to $y+u$. Assumption (i) and the assumed lower bound on $y$ 
guarantee that 
$$
\alpha \le c\frac{1+|y+u|^2}{|y+u|^2}\le 2c
$$ 
and Assumption (ii) implies 
$$
|\beta|=|\langle b(y+u),v\rangle| \le c(1+|y+u|^2).
$$
Therefore, using $\langle y,y+u\rangle\ge |y|^2-|y||u|\ge0$, and $\langle y+u,v\rangle=0$, we get
\begin{align*}
\langle y,b(y+u)\rangle &= \alpha \langle y,y+u\rangle+ \beta \langle y,v \rangle =  \alpha \langle y,y+u\rangle+ \beta \langle y+u,v \rangle -   \beta \langle u,v \rangle\\
&\le 2c\langle y,y+u\rangle +c(1+|y+u|^2)|u||v|\\ 
&\le |y|^2(2c+|u|c)+|y|(2c|u|+2c|u|^2+c(|u|^3+|u|)\le C(1+|y|^2)
\end{align*}
for an appropriate $C=C(T,\omega,\sigma,c)$ showing \eqref{to show} thus completing the proof of the proposition.
\end{proof}

\begin{remark} Note that the proof of the previous proposition does not use any properties of the Brownian motion $W$ other than almost sure 
local boundedness, so it can -- for example -- also be applied to additive L\'evy noise. It is of interest to compare Proposition \ref{addinoise} with 
\cite[Section 3.3]{CHJ13}. They provide an example for $d=2$ in which the drift has no radial component and the tangential component grows like the 
third power of the distance to the origin. That example is strongly complete but  there exists a continuous function $g$ starting at 0 such that the 
solution of \eqref{additive} with $W$ replaced by $g$ blows up for some initial conditions.
\end{remark}

\section{Appendix}

We use the notation $Z^*_T=\sup_{0 \le t \le T} Z_t$ for a real-valued process $Z$.

The following lemma is taken from \cite{RS10}.

\begin{lemma}\label{thefirstlemma}
Let $\sigma>0$ be a stopping time and let $Z$ be an adapted non-negative stochastic process with continuous paths defined
on $[0,\sigma[$ which satisfies the inequality
\begin{equation*}
Z_t \le \int_0^t \rho(Z^*_u)\, \dd u + M_t + C,
\end{equation*}
and $\lim_{t \uparrow \sigma} Z^*_t= \infty$ on $\{\sigma < \infty\}$ almost surely. Here, $C \ge 0$ and $M$ is a continuous
local martingale defined
on $[0,\sigma[$, $M_0=0$ and $\rho:[0,\infty[ \to ]0,\infty[$ is non-decreasing,  and $\int_0^{\infty} 1/\rho (u)\, \dd u = \infty$.
Then $\sigma = \infty$ almost surely.
\end{lemma}

The following stochastic Gronwall lemma is taken from \cite{S13} (for a more recent paper providing optimal constants see \cite{B13}). 
For $p \in (0,1)$ define 
$$\tilde c_p:=\big(4 \wedge \frac 1p\big) \frac{\pi p}{\sin(\pi p)}+1.$$
\begin{lemma}\label{thelem}
Let $Z$ and $H$ be nonnegative, adapted processes with continuous paths and assume that 
$\psi$ is nonnegative and progressively measurable. 
Let $M$ be a continuous local martingale starting at 0. If 
\begin{equation*}%\label{Gron}
Z_t \le \int_0^t \psi_s Z_s\,\dd s + M_t +H_t
\end{equation*}
holds for all $t \ge 0$, then for $p\in(0,1)$, and $\mu,\nu>1$ such that $\frac 1\mu + \frac 1\nu =1$ and $p\nu<1$, we have 
\begin{equation*}%\label{eins}
\E \sup_{0 \le s \le t} Z^p_s \le (\tilde c_{p\nu})^{1/\nu}  \Big(\E\exp\Big\{p\mu\int_0^t \psi_s\,\dd s\Big\}\Big)^{1/\mu}  \Big(\E (H^*_t)^{p\nu}\Big)^{1/{\nu}}.
\end{equation*}
If $\psi$ is deterministic, then
\begin{equation*}%\label{zwei}
\E \sup_{0 \le s \le t}\ Z^p_s \le \tilde c_p \exp\Big\{p\int_0^t \psi_s\,\dd s\Big\}  \Big(\E (H^*_t)^p\Big),
\end{equation*}
and
\begin{equation*}%\label{drei}
\E Z_t \le \exp\Big\{\int_0^t \psi_s\,\dd s\Big\} \E H^*_t.
\end{equation*}
\end{lemma}

We will need the following lemma at two different places.

\begin{lemma}\label{weiteres2}
Let $x,y \in \R^d$, $x\neq y$, $0=\gamma_0<\gamma_1<...<\gamma_n=1$  and define $x_i:=x+\gamma_i(y-x)$, $i=0,...,n$. Then
$$
\frac{\|\A (x,y)\|}{|x-y|} \le \sum_{i=1}^n \frac{\|\A (x_{i-1},x_{i})\|}{|x_{i-1}-x_i|},\qquad \frac{\tr\A (x,y)}{|x-y|} \le \sum_{i=1}^n \frac{\tr\A (x_{i-1},x_{i})}{|x_{i-1}-x_i|}.
$$
\end{lemma}

\begin{proof}
Let $v \in \R^d$ and $\alpha_1,...,\alpha_n>0$ such that $\sum_{i=1}^n \alpha_i=1$. Define $A_i(t):=\sum_{k=1}^n v_k(M_k(t,x_i)-M_k(t,x_{i-1}))$. 
Using Jensen's inequality, we get
\begin{align*}
\langle \A(x,y)v,v\rangle &=\frac{\dd}{\dd t}\Big[ \sum_{i=1}^n  A_i \Big]_t
\le \frac{\dd}{\dd t} \sum_{i=1}^n \frac 1{\alpha_i}\big[A_i\big]_t = \sum_{i=1}^n \frac{1}{\alpha_i}\langle  \A(x_i,x_{i-1})    v,v\rangle.
%[\sum_k v_k(M_k(t,x_i)-M_k(t,x_{i-1})),\sum_l v_l(M_l(t,x_j)-M_l(t,x_{j-1}))]\\
%&\le \sum_{i=1}^n \frac{\dd}{\dd t}[\sum_k v_k(M_k(t,x_i)-M_k(t,x_{i-1}))]+  
\end{align*}
Let $\alpha_i:=|x_i-x_{i-1}|/|x-y|$. Taking the supremum over all $v$ with norm 1, the first claim follows. Choosing $v$ to be the $j$-th unit vector and 
summing over $j=1,\cdots,d$, the second claim follows.
\end{proof}

\begin{lemma}\label{G}
Let $a$ and $b$ satisfy our general assumptions and assume that, in addition, there exists $K \ge 0$ such that
$$
2 \langle b(x)-b(y),x-y \rangle + \tr(\cA(x,y)) \le K |x-y|^2 \mbox{ for all } x,y \in \R^d.
$$
Then 
$$
2 \langle b(x),x \rangle + \tr(a(x,x)) \le K|x|^2 + \bigO (|x|).
$$
In particular, {\rm(A$_{0,K}$)} implies {\rm{(G}}$_{\rho})$ for a positive multiple $\rho$ of $\tilde \rho(x):=x \vee 1$.
\end{lemma}

\begin{proof}
Let $x \in \R^d \backslash \{0\}$. For $n \in \N$ and $0= \gamma_0< ... < \gamma_n=1$ and $x_i:=\gamma_i\, x$, we have 
\begin{align*}
2 \langle b(x),x \rangle &+ \tr(a(x,x))\\ 
&= 2  \langle b(0),x \rangle + \tr(a(x,x))   + 
2\sum_{i=1}^n \langle b(x_i)-b(x_{i-1}),x_i-x_{i-1}\rangle \frac{|x|}{|x_i-x_{i-1}|}\\
 &\le K\,|x|^2 +  2  \langle b(0),x \rangle + |x| \Big( \frac 1{|x|} \tr(a(x,x))   -  
\sum_{i=1}^n  \frac{\tr(\cA(x_i,x_{i-1}))}{|x_i-x_{i-1}|} \Big). 
\end{align*}
Therefore,
$$
2 \langle b(x),x \rangle + \tr(a(x,x))  \le K\,|x|^2 +  
2  \langle b(0),x \rangle + |x| \Big( \frac 1{|x|} \tr(a(x,x))   - H(0,x)\Big),
$$
where 
$$
H(y,z):=\sup  \Big\{ \sum_{i=1}^n  \frac{\tr(\cA(\xi_i,\xi_{i-1}))}{|\xi_i-\xi_{i-1}|} \Big\},
$$
where the supremum is extended over all partitions of the line segment from $y$ to  $z$. By Lemma \ref{weiteres2}, the sum in the definition of 
$H(x,y)$ is non-decreasing when the partition is refined and hence $H$ has the following {\em additivity property}: 
$H(y,z)=H(y,\xi)+H(\xi,z)$ whenever $\xi$ lies on the line segment from $y$ to $z$. Further,  for $\alpha,\beta>0$ such 
that $(1-\alpha)(1+\beta)=1$ and $y,z \in \R^d$, we have
\begin{align*}
0 &\le \Big(\sqrt{(1-\alpha)a_{i,i}(z,z)}-\sqrt{(1+\beta)a_{i,i}(y,y)}\Big)^2\\
&=  (1-\alpha)a_{i,i}(z,z) +(1+\beta)a_{i,i}(y,y)-2\sqrt{a_{i,i}(y,y)}\sqrt{a_{i,i}(z,z)}\\
&\le (1-\alpha)a_{i,i}(z,z) +(1+\beta)a_{i,i}(y,y) - a_{i,i}(y,z) - a_{i,i}(z,y),
\end{align*}
where we used the Kunita-Watanabe inequality in the last step.
Therefore,
$$
\tr (\cA(y,z)) \ge \alpha  \tr (a(z,z))-\beta \tr (a(y,y)).
$$
For $\gamma>1,\,y=x$, and $z=\gamma x$, we therefore get 
$$
\frac{\tr (a(\gamma x,\gamma x))}{\gamma |x|} - \frac{\tr(a(x,x))}{|x|} 
\le \frac{\tr(\cA(x,\gamma x))}{(\gamma-1)|x|} \le H(x,\gamma x).
$$ 
Using the additivity property of $H$, we see that the function 
$\gamma \mapsto  \frac{\tr(a(\gamma x,\gamma x))}{\gamma |x|} - H(0,\gamma x)$ is non-increasing.
Since $a$ and $b$ are locally bounded, we obtain
$$
2 \langle b(x),x \rangle + \tr (a(x,x)) \le K|x|^2 + \bigO (|x|),
$$
as required.
\end{proof}

Finally we show that for Assumption (A$_{\mu,K}$) to hold it suffices that it holds locally.

\begin{proposition}\label{locglob}
If 
$$
2\langle b(y)-b(x),y-x \rangle + \tr  \A(x,y) + \mu \| \A(x,y)\| \le C|y-x|^2
$$
holds locally, then it holds also globally.
\end{proposition}

\begin{proof}
Let $x,y \in \R^d$. For an equidistant partition $x_0=x,...,x_n=y$ of the straight line connecting $x$ and $y$, Lemma \ref{weiteres2} implies
\begin{align*}
2&\langle b(y)-b(x),y-x \rangle + \tr  \A(x,y)  + \mu \| \A(x,y)\|\\
%&=2\sum_{i=0}^{n-1}\langle b(x_{i+1})-b(x_i),y-x \rangle + \tr  \A(x,y)  + \mu \| \A(x,y)\|\\
&\le \sum_{i=0}^{n-1}\Big( 2\langle b(x_{i+1})-b(x_i),y-x \rangle +  n \big(\tr  \A(x_{i+1},x_i)  + \mu \| \A(x_{i+1},x_i)\|\big)\Big)\\
%&\le C|y-x|^2-n  \sum_{i=0}^{n-1}(\tr \A(x_i,x_{i+1})+\mu\|\A(x_i,x_{i+1})\|) +  \tr  \A(x,y)  + \mu \| \A(x,y)\|\\
&\le C|y-x|^2,
\end{align*}
provided the partition is fine enough. Therefore the assertion follows.
\end{proof}

\noindent {\bf Acknowledgement.} We thank Sebastian Riedel for valuable discussions related to strong completeness in the additive noise case.

\end{document}